\documentclass{article}%
\usepackage{amsfonts}
\usepackage{amsmath}
\usepackage{amssymb}
\usepackage{graphicx}%
\newtheorem{theorem}{Theorem}

\newtheorem{definition}[theorem]{Definition}

\newtheorem{lemma}[theorem]{Lemma}

\newtheorem{remark}[theorem]{Remark}

\newenvironment{proof}[1][Proof]{\noindent\textbf{#1.} }{\ \rule{0.5em}{0.5em}}
\begin{document}

\title{A general stochastic maximum principle for mixed relaxed-singular control
problems\thanks{This work is partially supported by Algerian-French
cooperation, Tassili 07 MDU 705.}}
\author{Seid Bahlali\\{\small Laboratory of Applied Mathematics}\\{\small \ University Med Khider}\\{\small \ \ Po. Box 145, Biskra 07000, Algeria}\\{\small \ Email : sbahlali@yahoo.fr}}
\maketitle

\begin{abstract}
We consider in this paper, mixed relaxed-singular stochastic control problems,
where the control variable has two components, the first being measure-valued
and the second singular. The control domain is not necessarily convex and the
system is governed by a nonlinear stochastic differential equation, in which
the measure-valued part of the control enters both the drift and the diffusion
coefficients. We establish necessary optimality conditions, of the Pontryagin
maximum principle type, satisfied by an optimal relaxed-singular control,
which exist under general conditions on the coefficients. The proof is based
on the strict singular stochastic maximum principle established by
Bahlali-Mezerdi, Ekeland's variational principle and some stability properties
of the trajectories and adjoint processes with respect to the control variable.

\ 


\ 


\end{abstract}

\section{Introduction}

In this paper we study a stochastic control problems of nonlinear systems,
where the control variable has two components, the first being measure-valued
and the second singular. The system is governed by a stochastic differential
equation (SDE\ for short) of the type
\[
\left\{
\begin{array}
[c]{l}%
dx_{t}^{q}=\int_{A_{1}}b\left(  t,x_{t}^{q},a\right)  q_{t}\left(  da\right)
dt+\int_{A_{1}}\sigma\left(  t,x_{t}^{q},a\right)  q_{t}\left(  da\right)
dW_{t}+G_{t}d\eta_{t},\\
x_{0}^{q}=x_{0}.
\end{array}
\right.
\]
where $b,\sigma$ and $G$ are given deterministic functions, $x_{0}$ is the
initial data and $W=\left(  W_{t}\right)  _{t\geq0}$ is a $d$-dimensional
standard Brownian motion, defined on a filtered probability space $\left(
\Omega,\mathcal{F},\left(  \mathcal{F}_{t}\right)  _{t\geq0},\mathcal{P}%
\right)  ,$ satisfying the usual conditions. The control variable is a
suitable process $\left(  q,\eta\right)  $\ where $q:\left[  0,T\right]
\times\Omega\longrightarrow\mathbb{P}\left(  A_{1}\right)  $, $\eta:\left[
0,T\right]  \times\Omega\longrightarrow A_{2}=\left(  [0,\infty)\right)  ^{m}%
$\ are $B\left[  0,T\right]  \otimes\mathcal{F}$-measurable, $\left(
\mathcal{F}_{t}\right)  $- adapted, and $\eta$\ is an increasing process
(componentwise), continuous on the left with limits on the right with
$\eta_{0}=0$.

The pair $\left(  q,\eta\right)  $ is called mixed relaxed-singular control
(relaxed control for short) and we denote by $\mathcal{R}$ the class of
relaxed controls.

The functional cost, to be minimized over $\mathcal{R}$, has the form%
\[
J\left(  q,\eta\right)  =\mathbb{E}\left[  g\left(  x_{T}^{q}\right)
+\int_{0}^{T}\int_{A_{1}}h\left(  t,x_{t}^{q},a\right)  q_{t}\left(
da\right)  dt+\int_{0}^{T}l_{t}d\eta_{t}\right]  .
\]

A relaxed control $\left(  \mu,\xi\right)  $ is called optimal if it satisfies%
\[
J\left(  \mu,\xi\right)  =\inf\limits_{(q,\eta)\in\mathcal{R}}J(q,\eta).
\]

Singular control problems have been studied by many authors including
Ben\u{e}s-Shepp-Witsenhausen $\left[  5\right]  ,$ Chow-Menaldi-Robin $\left[
8\right]  ,$ Karatzas-Shreve $\left[  18\right]  ,$ Davis-Norman $\left[
9\right]  ,$ Haussmann-Suo $\left[  14,15,16\right]  .$ See $\left[
15\right]  $ for a complete list of references on the subject. The approaches
used in these papers, to solve the problem are mainly based on dynamic
programming. It was shown in particular that the value function is solution of
a variational inequality, and the optimal state is a reflected diffusion at
the free boundary. Note that in $\left[  14\right]  ,$ the authors apply the
compactification method to show existence of an optimal relaxed-singular control.

The other major approach to solve control problems is to derive necessary
conditions satisfied by some optimal control, known as the stochastic maximum
principle. The first version of the stochastic maximum principle that covers
singular control problems was obtained by Cadenillas-Haussmann $\left[
7\right]  $, in which they consider linear dynamics, convex cost criterion and
convex state constraints. necessary optimality conditions for non linear SDEs
were obtained by Bahlali-Chala $\left[  1\right]  $ and Bahlali-Mezerdi
$\left[  2\right]  .$

The common fact in this works is that an optimal strict singular control does
not necessarily exist, the set $\mathcal{U}$\ of strict singular controls
$\left(  v,\eta\right)  $, where $v:\left[  0,T\right]  \times\Omega
\longrightarrow A_{1}\subset\mathbb{R}^{k}$, is too narrow and not being
equipped with a good topological structure. The idea is then to introduce the
class $\mathcal{R}$\ of relaxed controls in which the controller chooses at
time $t$, a probability measure $q_{t}\left(  da\right)  $\ on the set $A_{1}%
$, rather than an element $v_{t}$\ of $A_{1}$. The relaxed control problem
find its interest in two essential points. The first is that it is a
generalization of the strict control problem, indeed if $q_{t}\left(
da\right)  =\delta_{v_{t}}\left(  da\right)  $ is a Dirac measure concentrated
at a single point $v_{t}$, then we get a strict control problem as a
particular case of the relaxed one. The second is that an optimal relaxed
control exists.

Stochastic maximum principle for relaxed controls ( without the singular part)
was obtained by Mezerdi-Bahlali $\left[  21\right]  $ in the case of
uncontrolled diffusion and Bahlali-Mezerdi-Djehiche $\left[  3\right]  $ where
the drift and the diffusion coefficients depends explicitly on the relaxed
control variable. Necessary optimality conditions for relaxed-singular
controls and uncontrolled diffusion are derived by Bahlali-Djehiche-Mezerdi
$\left[  4\right]  $.

Our main goal in this paper, is to establish a maximum principle for
relaxed-singular controls, where the first part of the control is a
measure-valued process and enters both the drift and the diffusion
coefficients. This leads to necessary optimality conditions satisfied by an
optimal relaxed control, which exists under general conditions on the
coefficients ( see $\left[  14\right]  $). To achieve this goal, we use the
maximum principle for strict singular controls established by Bahlali-Mezerdi
$\left[  2\right]  $ and Ekeland's variational principle. We are able to prove
necessary conditions for near optimality satisfied by a sequence of strict
controls, converging in some sense to the relaxed optimal control, by the so
called chattering lemma. The relaxed maximum principle is then derived by
using some stability properties of the trajectories and the adjoint processes
with respect to the control variable.

This result generalizes at the same time the results of Bahlali-Mezerdi
$\left[  2\right]  $, Bahlali-Mezerdi-Djehiche $\left[  3\right]  $ and
Bahlali-Djehiche-Mezerdi $\left[  4\right]  $. We note that the result of
$\left[  2\right]  $ and $\left[  3\right]  $ are the extensions of the Peng's
stochastic maximum principle $\left[  22\right]  $ respectively to the
singular and relaxed controls.

The paper is organized as follows. In Section 2, we formulate the strict and
relaxed control problems and give the various assumptions used throughout the
paper. Section 3 is devoted to the proof of the main approximation. In Section
4, we establish the near stochastic maximum principle. In the last Section, we
state and prove the main result of this paper, which is the stochastic maximum
principle for relaxed-singular controls.

\ 

Along this paper, we denote by $C$ some positive constant and for simplicity,
we need the following matrix notations. We denote by $\mathcal{M}_{n\times
d}\left(  \mathbb{R}\right)  $ the space of $n\times d$ real matrix and
$\mathcal{M}_{n\times n}^{d}\left(  \mathbb{R}\right)  $ the linear space of
vectors $M=\left(  M_{1},...,M_{d}\right)  $ where $M_{i}\in\mathcal{M}%
_{n\times n}\left(  \mathbb{R}\right)  $. For any $M,N\in\mathcal{M}_{n\times
n}^{d}\left(  \mathbb{R}\right)  $, $L,S\in\mathcal{M}_{n\times d}\left(
\mathbb{R}\right)  $, $Q\in\mathcal{M}_{n\times n}\left(  \mathbb{R}\right)
$, $\alpha,\beta\in\mathbb{R}^{n}$ and $\gamma\in\mathbb{R}^{d},$ we use the
following notations

$\alpha\beta=%
{\displaystyle\sum\limits_{i=1}^{n}}
\alpha_{i}\beta_{i}\in\mathbb{R}$ is the product scalar in $\mathbb{R}^{n}$;

$LS=%
{\displaystyle\sum\limits_{i=1}^{d}}
L_{i}S_{i}\in\mathbb{R}$, where $L_{i}$ and\ $S_{i}$ are the $i^{th}$ columns
of $L$ and $S;$

$ML=%
{\displaystyle\sum\limits_{i=1}^{d}}
M_{i}L_{i}\in\mathbb{R}^{n}$;

$M\alpha\gamma=\sum\limits_{i=1}^{d}\left(  M_{i}\alpha\right)  \gamma_{i}%
\in\mathbb{R}^{n}$;

$MN=%
{\displaystyle\sum\limits_{i=1}^{d}}
M_{i}N_{i}\in\mathcal{M}_{n\times n}\left(  \mathbb{R}\right)  $;

$MQN=%
{\displaystyle\sum\limits_{i=1}^{d}}
M_{i}QN_{i}\in\mathcal{M}_{n\times n}\left(  \mathbb{R}\right)  $;

$MQ\gamma=%
{\displaystyle\sum\limits_{i=1}^{d}}
M_{i}Q\gamma_{i}\in\mathcal{M}_{n\times n}\left(  \mathbb{R}\right)  $.

We denote by $L^{\ast}$ the transpose of the matrix $L$ and $M^{\ast}=\left(
M_{1}^{\ast},...,M_{d}^{\ast}\right)  $.

\section{Formulation of the problem}

Let $\left(  \Omega,\mathcal{F},\left(  \mathcal{F}_{t}\right)  _{t\geq
0},\mathcal{P}\right)  $ be a filtered probability space satisfying the usual
conditions, on which a $d$-dimensional Brownian motion $W=\left(
W_{t}\right)  _{t\geq0}$\ is defined. We assume that $\left(  \mathcal{F}%
_{t}\right)  $ is the $\mathcal{P}$- augmentation of the natural filtration of
$W.$

Let $T$ be a strictly positive real number and consider the following sets

$A_{1}$ is a non empty subset of $\mathbb{R}^{k}$ and $A_{2}=\left(  \left[
0,\infty\right)  \right)  ^{m}.$

$U_{1}$ is the class of measurable, adapted processes $v:\left[  0,T\right]
\times\Omega\longrightarrow A_{1}.$

$U_{2}$ is the class of measurable, adapted processes $\eta:\left[
0,T\right]  \times\Omega\longrightarrow A_{2}$ such that $\eta$ is
nondecreasing (componentwise), left-continuous with right limits and $\eta
_{0}=0$.

\subsection{The strict control problem}

\begin{definition}
\textit{An admissible strict control is an }$\mathcal{F}_{t}$%
-\textit{\ adapted process }$\left(  v,\eta\right)  \in U_{1}\times U_{2}%
$\textit{\ such that }
\[
\mathbb{E}\left[  \underset{t\in\left[  0,T\right]  }{\sup}\left\vert
v_{t}\right\vert ^{2}+\left\vert \eta_{T}\right\vert ^{2}\right]  <\infty.
\]
\textit{We denote by }$\mathcal{U}$\textit{\ the set of all admissible
controls.}
\end{definition}

For any $\left(  v,\eta\right)  \in\mathcal{U}$, we consider the following
SDE
\begin{equation}
\left\{
\begin{array}
[c]{l}%
dx_{t}^{v}=b\left(  t,x_{t}^{v},v_{t}\right)  dt+\sigma\left(  t,x_{t}%
^{v},v_{t}\right)  dW_{t}+G_{t}d\eta_{t},\\
x_{0}^{v}=x_{0},
\end{array}
\right.
\end{equation}
where%
\[%
\begin{array}
[c]{l}%
b:\left[  0,T\right]  \times\mathbb{R}^{n}\times A_{1}\longrightarrow
\mathbb{R}^{n},\\
\sigma:\left[  0,T\right]  \mathbb{\times R}^{n}\times A_{1}\longrightarrow
\mathcal{M}_{n\times d}\left(  \mathbb{R}\right)  ,\\
G:\left[  0,T\right]  \longrightarrow\mathcal{M}_{n\times m}\left(
\mathbb{R}\right)  .
\end{array}
\]
The expected cost, to be minimized over the class $\mathcal{U}$, has the form
\begin{equation}
J\left(  v,\eta\right)  =\mathbb{E}\left[  g\left(  x_{T}^{v}\right)
+\int_{0}^{T}h\left(  t,x_{t}^{v},v_{t}\right)  dt+\int_{0}^{T}l_{t}d\eta
_{t}\right]  ,
\end{equation}
where%
\begin{align*}
g  &  :\mathbb{R}^{n}\longrightarrow\mathbb{R},\\
h  &  :\left[  0,T\right]  \times\mathbb{R}^{n}\times A_{1}\longrightarrow
\mathbb{R},\\
l  &  :\left[  0,T\right]  \longrightarrow\left(  \lbrack0,\infty)\right)
^{m}.
\end{align*}
A control $\left(  u,\xi\right)  \in\mathcal{U}$ is called optimal, if that
solves%
\begin{equation}
J\left(  u,\xi\right)  =\inf\limits_{\left(  v,\eta\right)  \in\mathcal{U}%
}J\left(  v,\eta\right)  ,
\end{equation}
The following assumptions will be in force throughout this paper
\begin{align}
&  b,\sigma,g,h\ \text{are twice continuously differentiable with respect to
}x\text{.}\\
&  \text{The derivatives }b_{x},b_{xx},\sigma_{x},\sigma_{xx},g_{x}%
,g_{xx},h_{x},h_{xx}\text{ are continuous}\nonumber\\
&  \text{in }\left(  x,v\right)  \text{ and uniformly bounded.}\nonumber\\
&  b,\sigma\text{ are bounded by }C\left(  1+\left\vert x\right\vert
+\left\vert v\right\vert \right)  \text{.}\nonumber\\
&  G\text{ and }l\text{ are continuous and }G\text{ is bounded.}\nonumber
\end{align}
Under the above hypothesis, for every $\left(  v,\eta\right)  \in\mathcal{U}$,
equation $\left(  1\right)  $ has a unique strong solution and the cost
functional $J$ is well defined from $\mathcal{U}$ into $\mathbb{R}$.

\subsection{The relaxed model}

The strict control problem $\left\{  \left(  1\right)  ,\left(  2\right)
,\left(  3\right)  \right\}  $ formulated in the last subsection may fail to
have an optimal solution. Let us begin by two deterministic examples who show
that even in simple cases, existence of a strict optimal control is not ensured.

\textbf{Example 1. }The problem is to minimize, over the set of measurable
functions $v:\left[  0,T\right]  \rightarrow\left\{  -1,1\right\}  $, the
following functional cost%
\[
J\left(  v\right)  =%
{\displaystyle\int\nolimits_{0}^{T}}
\left(  x_{t}^{v}\right)  ^{2}dt,
\]
where $x_{t}^{v}$ denotes the solution of%
\[
\left\{
\begin{array}
[c]{c}%
dx_{t}^{v}=v_{t}dt,\\
x_{0}^{v}=0.
\end{array}
\right.
\]

We have
\[
\inf\limits_{v\in\mathcal{U}}J\left(  v\right)  =0.
\]

Indeed, consider the following sequence of controls%
\[
v_{t}^{n}=\left(  -1\right)  ^{k}\text{ \ if \ }%
\genfrac{.}{.}{}{0}{k}{n}%
T\leq t\leq%
\genfrac{.}{.}{}{0}{k+1}{n}%
T\ \ ,\ \ 0\leq k\leq n-1.
\]

Then clearly
\begin{align*}
\left\vert x_{t}^{v^{n}}\right\vert  &  \leq%
\genfrac{.}{.}{}{0}{T}{n}%
,\\
\left\vert J\left(  v^{n}\right)  \right\vert  &  \leq%
\genfrac{.}{.}{}{0}{T^{3}}{n^{2}}%
.
\end{align*}

Which implies that
\[
\inf\limits_{v\in\mathcal{U}}J\left(  v\right)  =0.
\]

There is however no control $v$ such that $J\left(  v\right)  =0$. If this
would have been the case, then for every $t,\ x_{t}^{v}=0$. This in turn would
imply that $v_{t}=0$, which is impossible. The problem is that the sequence
$\left(  v^{n}\right)  $ has no limit in the space of strict controls. This
limit if it exists, will be the natural candidate for optimality. If we
identify $v_{t}^{n}$ with the Dirac measure $\delta_{v_{t}^{n}}\left(
da\right)  $ and set $q_{n}\left(  dt,dv\right)  =\delta_{v_{t}^{n}}\left(
dv\right)  dt$, we get a measure on $\left[  0,1\right]  \times U$. Then, the
sequence $\left(  q_{n}\left(  dt,dv\right)  \right)  _{n}$ converges weakly
to $%
\genfrac{.}{.}{}{0}{1}{2}%
dt.\left[  \delta_{-1}+\delta_{1}\right]  \left(  da\right)  $.

\textbf{Example 2. }Consider the control problem where the system is governed
by the SDE%
\[
\left\{
\begin{array}
[c]{l}%
dx_{t}=v_{t}dt+dW_{t},\\
x_{0}=0.
\end{array}
\right.
\]

The functional cost to be minimized is given by%
\[
J\left(  v\right)  =\mathbb{E}\int_{0}^{T}\left[  x_{t}^{2}+\left(
1-v_{t}^{2}\right)  ^{2}\right]  dt.
\]

$U=\left[  -1,1\right]  $ and $x,\ v,\ W$ are one dimensional. The control $v$
(open loop) is a measurable function from $\left[  0,T\right]  $ into $U$. We
assume that $\mathbb{E}\left[  x_{t}^{2}\right]  =\left(  \mathbb{E}\left[
x_{t}\right]  \right)  ^{2}$. The separation principle applies to this
example, the optimal control minimizes%
\[
\int_{0}^{T}\left[  \widehat{x}_{t}^{2}+\left(  1-v_{t}^{2}\right)
^{2}\right]  ,
\]
where $\widehat{x}_{t}=\mathbb{E}\left[  x_{t}\right]  $ satisfies%
\[
\left\{
\begin{array}
[c]{l}%
d\widehat{x}_{t}=v_{t}dt,\\
\widehat{x}_{0}=0.
\end{array}
\right.
\]

This problem has no optimal strict control. A relaxed solution is to let
\[
\mu_{t}=%
\genfrac{.}{.}{}{}{1}{2}%
\delta_{1}+%
\genfrac{.}{.}{}{}{1}{2}%
\delta_{-1},
\]
where $\delta_{a}$ is an Dirac measure concentrated at a single point $a.$

\ 

This suggests that the set $\mathcal{U}$ of strict controls is too narrow and
should be embedded into a wider class with a richer topological structure for
which the control problem becomes solvable. The idea of relaxed control is to
replace the absolutely continuous part $v_{t}$ of the strict control by a
$\mathbb{P}\left(  A_{1}\right)  $-valued process $\left(  q_{t}\right)  $,
where $\mathbb{P}\left(  A_{1}\right)  $ is the space of probability measures
on $A_{1}$ equipped with the topology of weak convergence.

\begin{definition}
\textit{A relaxed control is a pair }$\left(  q,\eta\right)  $\textit{\ of
processes such that }$\eta\in U_{2}$ and $q$\textit{\ is a }$\mathbb{P}%
(A_{1})$\textit{--valued process, progressively measurable with respect to
}$\left(  \mathcal{F}_{t}\right)  $\textit{\ and such that for each }%
$t$\textit{, }$1_{\left(  0,t\right]  }.q$\textit{\ is }$\mathcal{F}_{t}$--measurable.

We denote by $\mathcal{R=R}_{1}\times U_{2}$ the set of relaxed controls.
\end{definition}

For any $\left(  q,\eta\right)  \in\mathcal{R}$, we consider the following
relaxed SDE
\begin{equation}
\left\{
\begin{array}
[c]{l}%
dx_{t}^{q}=\int_{A_{1}}b\left(  t,x_{t}^{q},a\right)  q_{t}\left(  da\right)
dt+\int_{A_{1}}\sigma\left(  t,x_{t}^{q},a\right)  q_{t}\left(  da\right)
dW_{t}+G_{t}d\eta_{t},\\
x_{0}^{q}=x_{0}.
\end{array}
\right.
\end{equation}

The expected cost, to be minimized over the class $\mathcal{R}$ of relaxed
controls, is defined as follows%
\begin{equation}
J(q,\eta)=\mathbb{E}\left[  g\left(  x_{T}^{q}\right)  +\int_{0}^{T}%
\int_{A_{1}}h\left(  t,x_{t}^{q},a\right)  q_{t}(da)dt+\int_{0}^{T}l_{t}%
d\eta_{t}\right]  .
\end{equation}

A relaxed control $\left(  \mu,\xi\right)  $ is called optimal, if it
satisfies
\begin{equation}
J\left(  \mu,\xi\right)  =\inf\limits_{(q,\eta)\in\mathcal{R}}J(q,\eta).
\end{equation}

The set $U_{1}$ of absolutely component of strict controls is embedded into
the set $\mathcal{R}_{1}$ of measure-valued processes by the mapping
\[
\Psi:v\in U_{1}\longmapsto\Psi\left(  v\right)  _{t}\left(  da\right)
=\delta_{v_{t}}(da)\in\mathcal{R}_{1},
\]
where, $\delta_{v}$ is the Dirac measure at a single point $v$.

\ 

Throughout this paper we suppose moreover that
\begin{align}
&  b,\ \sigma\text{ and }h\text{ are bounded,}\\
&  A_{1}\text{ is compact.}\nonumber
\end{align}

Haussmann and Suo $\left[  14\right]  $ have proved that the relaxed control
problem admits an optimal solution under general conditions on the
coefficients. Indeed, by using a compactification method and under some mild
continuity hypotheses on the data, it is shown by purely probabilistic
arguments that an optimal control for the problem exists. Moreover, the value
function is shown to be Borel measurable. See Haussmann and Suo $\left[
14\right]  $, Section 3, page 925 to page 934 and essentially Theorem 3.8,
page 933. See also $\left[  11,13\right]  $ for a complete study of relaxed controls.

\begin{remark}
If we put
\begin{align*}
\overline{b}\left(  t,x_{t}^{q},q_{t}\right)   &  =\int_{A_{1}}b\left(
t,x_{t}^{q},a\right)  q_{t}\left(  da\right)  ,\\
\overline{\sigma}\left(  t,x_{t}^{q},q_{t}\right)   &  =\int_{A_{1}}%
\sigma\left(  t,x_{t}^{q},a\right)  q_{t}\left(  da\right)  ,\\
\overline{h}\left(  t,x_{t}^{q},q_{t}\right)   &  =\int_{A_{1}}h\left(
t,x_{t}^{q},a\right)  q_{t}\left(  da\right)  ,
\end{align*}
then equation $\left(  5\right)  $ becomes
\[
\left\{
\begin{array}
[c]{l}%
dx_{t}^{q}=\overline{b}\left(  t,x_{t}^{q},q_{t}\right)  dt+\overline{\sigma
}\left(  t,x_{t}^{q},q_{t}\right)  dW_{t}+G_{t}d\eta_{t},\\
x_{0}^{q}=x_{0},
\end{array}
\right.
\]
with a functional cost given by
\[
J(q,\eta)=\mathbb{E}\left[  g\left(  x_{T}^{q}\right)  +\int_{0}^{T}%
\overline{h}\left(  t,x_{t}^{q},q_{t}\right)  dt+\int_{0}^{T}l_{t}d\eta
_{t}\right]  .
\]

Hence by introducing relaxed controls, we have replaced $A_{1}$ by a larger
space $\mathbb{P}\left(  A_{1}\right)  $. We have gained the advantage that
$\mathbb{P}\left(  A_{1}\right)  $ is both compact and convex. Moreover, the
drift, the diffusion and the running cost coefficients are linear with respect
to the measure-valued process $q.$
\end{remark}

\begin{remark}
The coefficients $\overline{b}\ $and $\overline{\sigma}$ (defined in the above
remark) check respectively the same assumptions as $b$ and $\sigma$. Then,
under assumptions $\left(  4\right)  $, $\overline{b}\ $and $\overline{\sigma
}$ are uniformly Lipschitz and with linear growth. Then, by classical results
on SDEs, for every $q\in\mathcal{R}$, equation $\left(  5\right)  $ admits a
unique strong solution.

On the other hand, It is easy to see that $\overline{h}$ checks the same
assumptions as $h$. Then, the functional cost $J$ is well defined from
$\mathcal{R}$ into $\mathbb{R}$.
\end{remark}

\begin{remark}
If $q_{t}=\delta_{v_{t}}$ is an atomic measure concentrated at a single point
$v_{t}\in A_{1}$, then for each $t\in\left[  0,T\right]  $ we have
\begin{align*}
\int_{A_{1}}b\left(  t,x_{t}^{q},a\right)  q_{t}\left(  da\right)   &
=\int_{A_{1}}b\left(  t,x_{t}^{q},a\right)  \delta_{v_{t}}\left(  da\right)
=b\left(  t,x_{t}^{q},v_{t}\right)  ,\\
\int_{A_{1}}\sigma\left(  t,x_{t}^{q},a\right)  q_{t}\left(  da\right)   &
=\int_{A_{1}}\sigma\left(  t,x_{t}^{q},a\right)  \delta_{v_{t}}\left(
da\right)  =\sigma\left(  t,x_{t}^{q},v_{t}\right)  ,\\
\int_{A_{1}}h\left(  t,x_{t}^{q},a\right)  q_{t}\left(  da\right)   &
=\int_{A_{1}}h\left(  t,x_{t}^{q},a\right)  \delta_{v_{t}}\left(  da\right)
=h\left(  t,x_{t}^{q},v_{t}\right)  .
\end{align*}

In this case $x^{q}=x^{v}$, $J(q,\eta)=J(v,\eta)$ and we get a strict control
problem. So the problem of strict controls $\left\{  \left(  1\right)
,\left(  2\right)  ,\left(  3\right)  \right\}  $ is a particular case of
relaxed control problems $\left\{  \left(  5\right)  ,\left(  6\right)
,\left(  7\right)  \right\}  $.
\end{remark}

\begin{remark}
The relaxed control problems studied e.g. in El Karoui et al $\left[
11\right]  $ and Bahlali-Mezerdi-Djehiche $\left[  3\right]  $ is different to
ours, in that they relax the corresponding infinitesimal generator of the
state process, which leads to a martingale problem for which the state process
driven by an orthogonal martingale measure. In our setting the driving
martingale measure $q_{t}\left(  da\right)  dW_{t}$ is however not orthogonal.
See Ma-Yong $\left[  20\right]  $ for more details.
\end{remark}

\section{Approximation of trajectories}

The next lemma, known as the Chattering Lemma, tells us that any
measure-valued process is a weak limit of a sequence of absolutely continuous
processes. This lemma was proved for deterministic measures and then extended
to random measures in $\left[  13\right]  $.

\begin{lemma}
(Chattering\thinspace\thinspace Lemma).\thinspace\ \textit{Let }$(q_{t}%
)$\textit{\ be a predictable process with values in the space of probability
measures on }$A_{1}$\textit{. Then there exists a sequence of predictable
processes }$(u^{n})$\textit{\ with values in }$A_{1}$\textit{\ such that }%
\begin{equation}
\delta_{u_{t}^{n}}(da)\,dt\mathit{\ }\text{converges weakly to\textit{\ }%
}q_{t}(da)\,dt,\,\,\,\mathcal{P}-a.s.
\end{equation}

\end{lemma}

\begin{proof}
See Fleming $\left[  13\right]  .$
\end{proof}

The next lemma gives the stability of the controlled SDE with respect to the
control variable.

\begin{lemma}
Let $\left(  q,\eta\right)  \in\mathcal{R}$ be a relaxed control and $x^{q}%
$\textit{\ the corresponding trajectory. Then there exists a sequence }%
$(v^{n},\eta)_{n}\subset\mathcal{U}$ such that
\begin{equation}
\underset{n\rightarrow\infty}{\lim}\mathbb{E}\left[  \underset{t\in\left[
0,T\right]  }{\sup}\left\vert x_{t}^{n}-x_{t}^{q}\right\vert ^{2}\right]  =0,
\end{equation}%
\begin{equation}
\underset{n\rightarrow\infty}{\lim}J(v^{n},\eta)=J(q,\eta),
\end{equation}
where $x^{n}$ denotes the solution of equation $\left(  1\right)  $ associated
with $\left(  v^{n},\eta\right)  $.
\end{lemma}

\begin{proof}
i) Let $q$ be a relaxed control, then from the chattering lemma (lemma7),
there exists a sequence of strict controls $\left(  v^{n}\right)  _{n}$\ such
that\textit{\ }%
\[
\delta_{v_{t}^{n}}(da)\,dt\mathit{\ }\text{converges weakly to\textit{\ }%
}q_{t}(da)\,dt,\,\,\,\mathcal{P}-a.s.
\]
where $\delta_{v_{t}^{n}}$ is a Dirac measure concentrated at a single point
$v_{t}^{n}$

Let $x^{q}$ and $x^{n}$ be the trajectories of the system associated,
respectively, with $\left(  q,\eta\right)  $ and $\left(  v^{n},\eta\right)
$, and $t\in\left[  0,T\right]  $, then%
\begin{align*}
x_{t}^{n}-x_{t}^{q}  &  =\int_{0}^{t}\left[  \int_{A_{1}}b\left(  s,x_{s}%
^{n},a\right)  \delta_{v_{s}^{n}}\left(  da\right)  -\int_{A_{1}}b\left(
s,x_{s}^{q},a\right)  q_{s}\left(  da\right)  \right]  ds\\
&  +\int_{0}^{t}\left[  \int_{A_{1}}\sigma\left(  s,x_{s}^{n},a\right)
\delta_{v_{s}^{n}}\left(  da\right)  -\int_{A_{1}}\sigma\left(  s,x_{s}%
^{q},a\right)  q_{s}\left(  da\right)  \right]  dW_{s},
\end{align*}

does not depend on the singular part. Then, $\left(  10\right)  $ is proved by
using the results and the same proof that in Bahlali-Mezerdi-Djehiche $\left[
3\text{, Lemma 4.1, page 12}\right]  $.

ii) On the other hand, $\left(  11\right)  $ is proved in $\left[  3\text{,
Lemma 4.1, page 12}\right]  .$
\end{proof}

\begin{remark}
As a consequence, it is easy to see that the strict and relaxed optimal
control problems have the same value function.
\end{remark}

\section{Maximum principle for near optimal controls}

In this section we establish necessary condition of near optimality satisfied
by a sequence of nearly optimal strict controls. This result is based on
Ekeland's variational principle which is given by the following.

\begin{lemma}
\textit{Let }$\left(  E,d\right)  $\textit{\ be a complete metric space and
}$f:E\longrightarrow\overline{\mathbb{R}}$\textit{\ be lower-semicontinuous
and bounded from below. Given }$\varepsilon>0$, suppose $u^{\varepsilon}\in E$
satisfies $f\left(  u^{\varepsilon}\right)  \leq\inf\left(  f\right)
+\varepsilon.$ Then for any $\lambda>0$, \textit{there exists }$v\in
E$\textit{\ such that }

\begin{enumerate}
\item $f\left(  v\right)  \leq f\left(  u^{\varepsilon}\right)  .$

\item $d\left(  u^{\varepsilon},v\right)  \leq\lambda.$

\item $f\left(  v\right)  <f\left(  w\right)  +\frac{\varepsilon}{\lambda
}d\left(  v,w\right)  \;,\;\forall\;w\neq v.$
\end{enumerate}
\end{lemma}

\begin{proof}
See Ekeland $\left[  10\right]  .$
\end{proof}

To apply Ekeland's variational principle, we have to endow the set
$\mathcal{U}$ of strict controls with an appropriate metric. For any $\left(
u,\xi\right)  ,\left(  v,\eta\right)  \in\mathcal{U}$, we set
\[%
\begin{array}
[c]{l}%
d_{1}\left(  u,v\right)  =\mathcal{P}\otimes dt\left\{  \left(  \omega
,t\right)  \in\Omega\times\left[  0,T\right]  ,\;u\left(  t,\omega\right)
\neq v\left(  t,\omega\right)  \right\}  ,\\
d_{2}\left(  \xi,\eta\right)  =\left(  \mathbb{E}%
{\displaystyle\int\nolimits_{0}^{T}}
\underset{t\in\left[  0,T\right]  }{\sup}\left\vert \xi_{t}-\eta
_{t}\right\vert ^{2}dt\right)  ^{1/2},\\
d\left[  \left(  u,\xi\right)  ,\left(  v,\eta\right)  \right]  =d_{1}\left(
u,v\right)  +d_{2}\left(  \xi,\eta\right)  ,
\end{array}
\]
where $\mathcal{P}\otimes dt$ is the product measure of $\mathcal{P}$ with the
Lebesgue measure $dt$.

Let us summarize some of the properties satisfied by $d.$

\begin{lemma}
\begin{enumerate}
\item $\left(  \mathcal{U},d\right)  $\textit{\ is a complete metric space.}

\item The cost functional $J$ is continuous from $\mathcal{U}$ into
$\mathbb{R}$.
\end{enumerate}
\end{lemma}

\begin{proof}
1. It is clear that $\left(  U_{2},d_{2}\right)  $ is a complete metric space.
Moreover, it was shown in $\left[  12\right]  $ that $\left(  U_{1}%
,d_{1}\right)  $ is a complete metric space. Hence $\left(  \mathcal{U}%
,d\right)  $ is a complete metric space as product of two complete metric spaces.

2. is proved in $\left[  12\right]  .$
\end{proof}

Now let $\left(  \mu,\xi\right)  \in\mathcal{R}$ be an optimal relaxed control
and denote by $x^{\mu}$ the trajectory of the system controlled by $\left(
\mu,\xi\right)  $. From $\left(  9\right)  $ and $\left(  10\right)  $, there
exists a sequence $\left(  u^{n}\right)  _{n}$ in $U_{1}$ such that
\begin{align*}
dt\mu_{t}^{n}(da)  &  =dt\delta_{u_{t}^{n}}(da)\underset{n\longrightarrow
\infty}{\longrightarrow}dt\mu_{t}\left(  da\right)  \text{ weakly}%
,\text{\textit{\ \ }}\mathcal{P}-a.s,\\
&  \mathbb{E}\left[  \underset{t\in\left[  0,T\right]  }{\sup}\left\vert
x_{t}^{n}-x_{t}^{\mu}\right\vert ^{2}\right]  \underset{n\longrightarrow
\infty}{\longrightarrow}0,
\end{align*}
where $x_{t}^{n}$ is the solution of equation $\left(  5\right)  $
corresponding to the control $\left(  \mu^{n},\xi\right)  .$

According to the optimality of $\left(  \mu,\xi\right)  $ and $\left(
9\right)  $, there exists a sequence $\left(  \varepsilon_{n}\right)  $ of
positive real numbers with $\underset{n\rightarrow\infty}{\lim}\varepsilon
_{n}=0$ such that%
\[
J\left(  u^{n},\xi\right)  =J\left(  \mu^{n},\xi\right)  \leq J\left(  \mu
,\xi\right)  +\varepsilon_{n}.
\]

A suitable version of lemma 10 $\left(  \text{see }\left[  12\right]  \text{
theorem 4.1}\right)  $ implies that a given any $\varepsilon_{n}>0$, there
exists $\left(  u^{n},\xi\right)  \in\mathcal{U}$ such that
\begin{align}
J\left(  u^{n},\xi\right)   &  \leq\underset{\left(  v,\eta\right)
\in\mathcal{U}}{\inf}J\left(  v,\eta\right)  +\varepsilon_{n},\\
J\left(  u^{n},\xi\right)   &  \leq J\left(  v,\eta\right)  +\varepsilon_{n}d
\left[  \left(  u^{n},\xi\right)  ;\left(  v,\eta\right)  \right]
\;;\;\forall\left(  v,\eta\right)  \in\mathcal{U}\text{.}\nonumber
\end{align}

Define the following two perturbations%
\begin{align}
\left(  u_{t}^{n,\theta},\xi_{t}\right)   &  =\left\{
\begin{array}
[c]{l}%
\left(  v,\xi_{t}\right)  \text{ \ if }t\in\left[  \tau,\tau+\theta\right]
,\\
\left(  u_{t}^{n},\xi_{t}\right)  \text{ \ Otherwise,}%
\end{array}
\right. \\
\left(  u_{t}^{n},\xi_{t}^{\theta}\right)   &  =\left(  u_{t}^{n},\xi
_{t}+\theta\left(  \eta_{t}-\xi_{t}\right)  \right)  ,
\end{align}
where $v$ is a $A_{1}$-valued, $\mathcal{F}_{t}$-measurable random variable
and $\eta$ is an increasing process with $\eta_{0}=0$ such that $\mathbb{E}%
\left[  \left\vert v_{t}\right\vert ^{2}+\left\vert \eta_{T}\right\vert
^{2}\right]  <\infty..$

From $\left(  12\right)  $ we have
\begin{align*}
0  &  \leq J\left(  u_{t}^{n,\theta},\xi_{t}\right)  -J\left(  u^{n}%
,\xi\right)  +\varepsilon_{n}d\left[  \left(  u^{n},\xi\right)  ;\left(
u^{n,\theta},\xi\right)  \right]  ,\\
0  &  \leq J\left(  u_{t}^{n},\xi_{t}^{\theta}\right)  -J\left(  u^{n}%
,\xi\right)  +\varepsilon_{n}d\left[  \left(  u^{n},\xi\right)  ;\left(
u^{n},\xi^{\theta}\right)  \right]  .
\end{align*}

From the definition of the metric $d,$ we obtain
\begin{align*}
0  &  \leq J\left(  u_{t}^{n,\theta},\xi_{t}\right)  -J\left(  u^{n}%
,\xi\right)  +\varepsilon_{n}d_{1}\left(  u^{n},u^{n,\theta}\right)  ,\\
0  &  \leq J\left(  u_{t}^{n},\xi_{t}^{\theta}\right)  -J\left(  u^{n}%
,\xi\right)  +\varepsilon_{n}d_{2}\left(  \xi,\xi^{\theta}\right)  .
\end{align*}

Using the definitions of $d_{1}$ and $d_{2}$, it holds that%
\begin{align}
0  &  \leq J\left(  u_{t}^{n,\theta},\xi_{t}\right)  -J\left(  u^{n}%
,\xi\right)  +\varepsilon_{n}C\theta,\\
0  &  \leq J\left(  u_{t}^{n},\xi_{t}^{\theta}\right)  -J\left(  u^{n}%
,\xi\right)  +\varepsilon_{n}C\theta.
\end{align}

From these above inequalities, we shall establish the near maximum principle
in integral form.

\begin{theorem}
(The near maximum principle in integral form). For each $\varepsilon_{n}%
>0$\textit{, there exists }$\left(  u^{n},\xi\right)  \in\mathcal{U}$
\textit{such that there exists two unique couples of adapted processes}%
\begin{align*}
\left(  p^{n},P^{n}\right)   &  \in L^{2}\left(  \left[  0,T\right]
;\mathbb{R}^{n}\right)  \times L^{2}\left(  \left[  0,T\right]  ;\mathbb{R}%
^{n\times d}\right)  ,\\
\left(  k^{n},K^{n}\right)   &  \in L^{2}\left(  \left[  0,T\right]
;\mathbb{R}^{n\times n}\right)  \times\left(  L^{2}\left(  \left[  0,T\right]
;\mathbb{R}^{n\times n}\right)  \right)  ^{d},
\end{align*}
\textit{solution of the following backward stochastic differential equations}%
\begin{equation}
\left\{
\begin{array}
[c]{l}%
-dp_{t}^{n}=H_{x}\left(  x_{t}^{n},u_{t}^{n},p_{t}^{n},P_{t}^{n}\right)
dt-P_{t}^{n}dW_{t}\\
p_{T}^{n}=g_{x}(x_{T}^{n}),
\end{array}
\right.
\end{equation}%
\begin{equation}
\left\{
\begin{array}
[c]{ll}%
-dk_{t}^{n}= & \left[  b_{x}^{\ast}\left(  t,x_{t}^{n},u_{t}^{n}\right)
k_{t}^{n}+k_{t}^{n}b_{x}\left(  t,x_{t}^{n},u_{t}^{n}\right)  \right]  dt\\
& +\sigma_{x}^{\ast}\left(  t,x_{t}^{n},u_{t}^{n}\right)  k_{t}^{n}\sigma
_{x}\left(  t,x_{t}^{n},u_{t}^{n}\right)  dt\\
& +\left[  \sigma_{x}^{\ast}\left(  t,x_{t}^{n},u_{t}^{n}\right)  K_{t}%
^{n}+K_{t}^{n}\sigma_{x}\left(  t,x_{t}^{n},u_{t}^{n}\right)  \right]  dt\\
& +H_{xx}\left(  x_{t}^{n},u_{t}^{n},p_{t}^{n},P_{t}^{n}\right)  dt-K_{t}%
^{n}dW_{t},\\
k_{T}^{n}= & g_{xx}(x_{T}^{n}).
\end{array}
\right.
\end{equation}
\textit{such that for all }$\left(  v,\eta\right)  \in\mathcal{U}$,%
\begin{gather}
H\left(  t,x_{t}^{n},u_{t}^{n},p_{t}^{n},P_{t}^{n}-k_{t}^{n}\sigma\left(
t,x_{t}^{n},u_{t}^{n}\right)  \right)  +%
\genfrac{.}{.}{}{0}{1}{2}%
Tr\left[  \sigma\sigma^{\ast}\left(  t,x_{t}^{n},u_{t}^{n}\right)  \right]
k_{t}^{n}\\
\leq H\left(  t,x_{t}^{n},v,p_{t}^{n},P_{t}^{n}-k_{t}^{n}\sigma\left(
t,x_{t}^{n},u_{t}^{n}\right)  \right)  +%
\genfrac{.}{.}{}{0}{1}{2}%
Tr\left[  \sigma\sigma^{\ast}\left(  t,x_{t}^{n},v\right)  \right]  k_{t}%
^{n}+C\varepsilon_{n}.\nonumber
\end{gather}%
\begin{equation}
0\leq\mathbb{E}\int_{0}^{T}\left(  l_{t}+G_{t}^{\ast}p_{t}^{n}\right)
d\left(  \eta-\xi\right)  _{t}+C\varepsilon_{n},
\end{equation}
where the Hamiltonian $H$ is defined from $\left[  0,T\right]  \times
\mathbb{R}^{n}\times A_{1}\times\mathbb{R}^{n}\times\mathcal{M}_{n\times
d}\left(  \mathbb{R}\right)  $ into $\mathbb{R}$ by
\[
H\left(  t,x,v,p,P\right)  =h\left(  t,x,v\right)  +pb\left(  t,x,v\right)
+\sigma\left(  t,x,v\right)  P.
\]

\end{theorem}

\begin{proof}
From inequalities $\left(  15\right)  $ and $\left(  16\right)  $, we use
respectively the same method as in $\left[  2\right]  $ with index $n$.
\end{proof}

\section{The relaxed stochastic maximum principle}

For simplicity, we note by $f\left(  t,x_{t}^{\mu},\mu_{t}\right)
=\int_{A_{1}}f\left(  t,x_{t}^{\mu},a\right)  \mu_{t}\left(  da\right)  $,
where $f$ stands for $b_{x},\sigma_{x},h_{x},H_{x},H_{xx}.$

Let $\left(  \mu,\xi\right)  $ be an optimal relaxed control and $x^{\mu}$ be
the corresponding optimal trajectory. Let $\left(  p^{\mu},P^{\mu}\right)  $
and $\left(  k^{\mu},K^{\mu}\right)  $ be the solutions of the following
backward stochastic differential equations%
\begin{equation}
\left\{
\begin{array}
[c]{l}%
-dp_{t}^{\mu}=H_{x}\left(  x_{t}^{\mu},\mu_{t},p_{t}^{\mu},P_{t}^{\mu}\right)
dt-P_{t}^{\mu}dW_{t},\\
p_{T}^{\mu}=g_{x}(x_{T}^{\mu}),
\end{array}
\right.
\end{equation}%
\begin{equation}
\left\{
\begin{array}
[c]{ll}%
-dk_{t}^{\mu}= & \left[  b_{x}^{\ast}\left(  t,x_{t}^{\mu},\mu_{t}\right)
k_{t}^{\mu}+k_{t}^{\mu}b_{x}\left(  t,x_{t}^{\mu},\mu_{t}\right)  \right]
dt\\
& +\sigma_{x}^{\ast}\left(  t,x_{t}^{\mu},\mu_{t}\right)  k_{t}^{\mu}%
\sigma_{x}\left(  t,x_{t}^{\mu},\mu_{t}\right)  dt\\
& +\left[  \sigma_{x}^{\ast}\left(  t,x_{t}^{\mu},\mu_{t}\right)  K_{t}^{\mu
}+K_{t}^{\mu}\sigma_{x}\left(  t,x_{t}^{\mu},\mu_{t}\right)  \right]  dt\\
& +H_{xx}\left(  x_{t}^{\mu},\mu_{t},p_{t}^{\mu},P_{t}^{\mu}\right)
dt-K_{t}^{\mu}dW_{t},\\
k_{T}^{\mu}= & g_{xx}(x_{T}^{\mu}),
\end{array}
\right.
\end{equation}

\begin{lemma}
We have
\begin{equation}
\underset{n\rightarrow\infty}{\lim}\left(  \mathbb{E}\left[  \underset
{t\in\left[  0,T\right]  }{\sup}\left\vert p_{t}^{n}-p_{t}^{\mu}\right\vert
^{2}\right]  +\mathbb{E}\int_{0}^{T}\left\vert P_{t}^{n}-P_{t}^{\mu
}\right\vert ^{2}ds\right)  =0.
\end{equation}%
\begin{equation}
\underset{n\rightarrow\infty}{\lim}\left(  \mathbb{E}\left[  \underset
{t\in\left[  0,T\right]  }{\sup}\left\vert k_{t}^{n}-k_{t}^{\mu}\right\vert
^{2}\right]  +\mathbb{E}\int_{0}^{T}\left\vert K_{t}^{n}-K_{t}^{\mu
}\right\vert ^{2}ds\right)  =0.
\end{equation}
where $\left(  p_{t}^{n},P_{t}^{n}\right)  $ and $\left(  k_{t}^{n},K_{t}%
^{n}\right)  $ are respectively the solutions of $\left(  17\right)  $ and
$\left(  18\right)  $.
\end{lemma}

\begin{proof}
We have the same proof that in Bahlali-Mezerdi-Djehiche $\left[  3,\text{
Lemma 4.11, page 19}\right]  $.
\end{proof}

\ 

We can now state the relaxed maximum principle in integral form.

\begin{theorem}
(The relaxed maximum principle in integral form). \textit{Let }$\left(
\mu,\xi\right)  $ \textit{be an optimal relaxed control minimizing the cost
}$J$\textit{\ over }$\mathcal{R}$ \textit{and }$x_{t}^{\mu}$\textit{\ the
corresponding optimal trajectory. there exists two unique couples of adapted
processes}%
\begin{align*}
\left(  p^{\mu},P^{\mu}\right)   &  \in L^{2}\left(  \left[  0,T\right]
;\mathbb{R}^{n}\right)  \times L^{2}\left(  \left[  0,T\right]  ;\mathbb{R}%
^{n\times d}\right)  ,\\
\left(  k^{\mu},K^{\mu}\right)   &  \in L^{2}\left(  \left[  0,T\right]
;\mathbb{R}^{n\times n}\right)  \times\left(  L^{2}\left(  \left[  0,T\right]
;\mathbb{R}^{n\times n}\right)  \right)  ^{d},
\end{align*}
\textit{which are respectively solutions of backward stochastic differential
equations }$\left(  21\right)  $\textit{\ and }$\left(  22\right)  $
\textit{such that}%
\begin{equation}
\mathcal{H}\left(  t,x_{t}^{\mu},\mu_{t},p^{\mu},P^{\mu},k^{\mu},K^{\mu
}\right)  =\underset{q\in\mathbb{P}\left(  A_{1}\right)  }{\inf}%
\mathcal{H}\left(  t,x_{t}^{\mu},q,p^{\mu},P^{\mu},k^{\mu},K^{\mu}\right)  ,
\end{equation}%
\begin{equation}
0\leq\mathbb{E}\int_{0}^{T}\left[  l_{t}+G_{t}^{\ast}p_{t}^{\mu}\right]
d\left(  \eta-\xi\right)  _{t}.
\end{equation}

Where
\begin{align*}
\mathcal{H}\left(  t,x_{t}^{\mu},\mu_{t},p^{\mu},P^{\mu},k^{\mu},K^{\mu
}\right)   &  =H\left(  t,x_{t}^{\mu},\mu_{t},p_{t}^{\mu},P_{t}^{\mu}%
-k_{t}^{\mu}\sigma\left(  t,x_{t}^{\mu},\mu_{t}\right)  \right) \\
&  +%
\genfrac{.}{.}{}{0}{1}{2}%
Tr\left[  \sigma\sigma^{\ast}\left(  t,x_{t}^{\mu},\mu_{t}\right)  \right]
k_{t}^{\mu}.
\end{align*}

\end{theorem}

\begin{proof}
Let $\left(  \mu,\xi\right)  $ be an optimal relaxed control. From theorem 12,
there exists a sequence $\left(  u^{n},\xi\right)  $ in $\mathcal{U}$ such
that for all $\left(  v,\eta\right)  \in\mathcal{U}$, the variational
equations $\left(  19\right)  $ and $\left(  20\right)  $ holds. Then by using
$\left(  10\right)  ,\ \left(  23\right)  $, $\left(  24\right)  $ and by
letting $n $ go to infinity, the results follows immediately.
\end{proof}

\ 

We are ready now state the main result of this paper, which is the relaxed
stochastic maximum principle for singular control problems in its global form.

\begin{theorem}
(The relaxed maximum principle in global form). \textit{Let }$\left(  \mu
,\xi\right)  $ \textit{be an optimal control minimizing the functional cost
}$J$ over $\mathcal{R}$\textit{\ and }$x_{t}^{\mu}$\textit{\ the trajectory of
the system controlled by }$\left(  \mu,\xi\right)  $\textit{. Then, there
exist two unique couples of adapted processes}%
\begin{align*}
\left(  p^{\mu},P^{\mu}\right)   &  \in L^{2}\left(  \left[  0,T\right]
;\mathbb{R}^{n}\right)  \times L^{2}\left(  \left[  0,T\right]  ;\mathbb{R}%
^{n\times d}\right)  ,\\
\left(  k^{\mu},K^{\mu}\right)   &  \in L^{2}\left(  \left[  0,T\right]
;\mathbb{R}^{n\times n}\right)  \times\left(  L^{2}\left(  \left[  0,T\right]
;\mathbb{R}^{n\times n}\right)  \right)  ^{d},
\end{align*}
\textit{which are respectively solutions of backward stochastic differential
equations }$\left(  21\right)  $\textit{\ and }$\left(  22\right)  $
\textit{such that}%
\begin{equation}
\mathcal{H}\left(  t,x_{t}^{\mu},\mu_{t},p^{\mu},P^{\mu},k^{\mu},K^{\mu
}\right)  =\underset{q\in\mathbb{P}\left(  A_{1}\right)  }{\inf}%
\mathcal{H}\left(  t,x_{t}^{\mu},q,p^{\mu},P^{\mu},k^{\mu},K^{\mu}\right)  ,
\end{equation}%
\begin{equation}
\mathcal{P}\left\{  \forall t\in\left[  0,T\right]  ,\;\forall i\;;\;l_{i}%
\left(  t\right)  +G_{i}^{\ast}\left(  t\right)  .p_{t}^{\mu}\geq0\right\}
=1,
\end{equation}%
\begin{equation}
\mathcal{P}\left\{  \sum_{i=1}^{m}\mathbf{1}_{\left\{  l_{i}\left(  t\right)
+G_{i}^{\ast}\left(  t\right)  .p_{t}^{\mu}\geq0\right\}  }d\xi_{t}%
^{i}=0\right\}  =1.
\end{equation}

\end{theorem}

\begin{proof}
From $\left(  25\right)  $, we deduce immediately $\left(  27\right)  $ and
from $\left(  26\right)  $, assertions $\left(  28\right)  $ and $\left(
29\right)  $ are proved exactly as in $\left[  7\right]  $.
\end{proof}

\begin{remark}
1) If $G=l=0$, we recover the relaxed stochastic maximum principle for
classical controls, see Bahlali-Mezerdi-Djehiche $\left[  3\right]  $.

2) If $\mu_{t}(da)=\delta_{u(t)}(da)$, we recover the strict singular
stochastic maximum principle established by Bahlali-Mezerdi $\left[  2\right]
.$

3) If $\mu_{t}(da)=\delta_{u(t)}(da)$ and $G=l=0$, we obtain Peng's stochastic
maximum principle $\left[  22\right]  $.

4) If the diffusion coefficient $\sigma$ does not contain the measure-valued
part, we recover the result of Bahlali-Djehiche-Mezerdi $\left[  4\right]  .$
\end{remark}

\ 

\textbf{Acknowledgement. }The authors would like to thank the referees for
valuable remarks and suggestions that improved the first version of the paper.


\begin{thebibliography}{99}                                                                                               %


\bibitem {}S. Bahlali and A. Chala, \textit{The stochastic maximum principle
in optimal control of singular diffusions with non linear coefficients}, Rand.
Operat. and Stoch. Equ, Vol. 18, 2005, no 1, pp 1-10.

\bibitem {}S. Bahlali and B. Mezerdi, \textit{A general stochastic maximum
principle for singular control problems, }Elect. J. of Probability, Vol. 10,
2005, Paper no 30, pp 988-1004.

\bibitem {}S. Bahlali, B. Mezerdi and B. Djehiche, \textit{Approximation and
optimality necessary conditions in relaxed stochastic control problems,}
Journal of Applied Mathematics and Stochastic Analysis, Volume 2006, pp 1-23.

\bibitem {}S. Bahlali, B. Djehiche and B. Mezerdi, \textit{The relaxed maximum
principle in singular control of diffusions}, SIAM J. Control and Optim, 2007,
Vol 46, Issue 2, pp 427-444.

\bibitem {}V.E Ben\u{e}s, L.A Shepp and H.S\ Witsenhausen, \textit{Some
solvable stochastic control problems}, Stochastics, 4 (1980), pp. 39-83.

\bibitem {}A. Bensoussan, \textit{Lecture on stochastic control,
}in\textit{\ }non linear filtering and stochastic control, Lecture notes in
mathematics 972, 1981, Proc. Cortona, Springer Verlag.

\bibitem {}A. Cadenillas and U.G. Haussmann, \textit{The stochastic maximum
principle for a singular control problem. }Stochastics and Stoch. Reports.,
Vol. 49, 1994, pp. 211-237.

\bibitem {}P.L. Chow, J.L. Menaldi and M. Robin, \textit{Additive control of
stochastic linear system with finite time horizon}, SIAM\ J. Control and
Optim., 23, 1985, pp. 858-899.

\bibitem {}M.H.A. Davis and A.\ Norman, \textit{Portfolio selection with
transaction costs}, Math. Oper. Research, 15, 1990, pp. 676 - 713.

\bibitem {}I. Ekeland, \textit{On the variational principle.} J.\ Math. Anal.
Appl., Vol. 47, 1974, pp 324-353..

\bibitem {}N. El Karoui, N. Huu Nguyen and M. Jeanblanc
Piqu\'{e},\textit{\ Compactification methods in the control of degenerate
diffusions.} Stochastics, Vol. 20, 1987, pp 169-219.

\bibitem {}R.J. Elliott and M. Kohlmann, \textit{The variational principle and
stochastic optimal control. }Stochastics, 1980, 3, pp 229-241.

\bibitem {}W.H. Fleming, \textit{Generalized solutions in optimal stochastic
control}, Differential games and control theory 2, (Kingston conference 1976),
Lect. Notes in Pure and Appl. Math.30, 1978.

\bibitem {}U.G. Haussmann and W. Suo, \textit{Singular optimal stochastic
controls I: Existence}, SIAM\ J. Control and Optim., Vol. 33, 1995, pp. 916-936.

\bibitem {}U.G. Haussmann and W. Suo, \textit{Singular optimal stochastic
controls II: Dynamic programming,} SIAM\ J. Control and Optim., Vol. 33, 1995,
pp 937-959.

\bibitem {}U.G. Haussmann and W. Suo, \textit{Existence of singular optimal
control laws for stochastic differential equations}, Stochastics and Stoch.
Reports, 48, 1994, pp 249 - 272.

\bibitem {}J. Jacod and J. M\'{e}min, \textit{Sur un type de convergence
interm\'{e}diaire entre la convergence en loi et la convergence en
probabilit\'{e}}. Sem. Proba.XV, Lect. Notes in Math. 851, Springer Verlag, 1980.

\bibitem {}I. Karatzas and S. Shreve, \textit{Connections between optimal
stopping and stochastic control I: Monotone follower problem}, SIAM\ J.
Control Optim., 22, 1984, pp 856 - 877.

\bibitem {}H.J. Kushner, \textit{Necessary conditions for continuous parameter
stochastic optimization problems,} SIAM J. Control Optim., Vol. 10, 1973, pp 550-565.

\bibitem {}J. Ma and J. Yong, \textit{Solvability of forward-backward SDEs and
the nodal set of Hamilton-Jacobi-Bellman equations}. A Chinese summary appears
in Chinese Ann. Math. Ser. A 16 (1995), no. 4, 532. Chinese Ann. Math. Ser. B
16, 1995, no. 3, pp 279--298.

\bibitem {}B. Mezerdi and\ S. Bahlali,\ \textit{Necessary conditions for
optimality in relaxed stochastic control problems},\ Stochastics And Stoch.
Reports, 2002, Vol 73 (3-4), pp 201-218.

\bibitem {}S. Peng, \textit{A general stochastic maximum principle for optimal
control problems}, SIAM Jour.Cont. Optim, 1990, 28, N${}^{\circ}$ 4, pp 966-979.

\bibitem {}J.\ Yong and X.Y. Zhou, \textit{Stochastic controls, Hamilton
systems and HJB\ equations}, vol 43, Springer, New York, 1999.
\end{thebibliography}
\end{document}